\documentclass[final,3p]{elsarticle}
 \usepackage{graphics}
 \usepackage{graphicx}
 \usepackage{epsfig}
\usepackage{amssymb}
 \usepackage{amsthm}
 \usepackage{lineno}
\usepackage{amsmath}
   \numberwithin{equation}{section}
\usepackage{mathrsfs}

\newtheorem{thm}{Theorem}[section]

\newtheorem{lem}[thm]{Lemma}
\newtheorem{prop}[thm]{Proposition}
\newtheorem{defn}[thm]{Definition}

 \biboptions{sort&compress,square}
\journal{}
\begin{document}
\begin{frontmatter}
\author[rvt1]{Jian Wang}
\ead{wangj484@nenu.edu.cn}
\author[rvt2]{Yong Wang\corref{cor2}}
\ead{wangy581@nenu.edu.cn}
\cortext[cor2]{Corresponding author.}

\address[rvt1]{School of Science, Tianjin University of Technology and Education, Tianjin, 300222, P.R.China}
\address[rvt2]{School of Mathematics and Statistics, Northeast Normal University,
Changchun, 130024, P.R.China}

\title{On K-K-W Type Theorems for Conformal Perturbations \\ of Twisted Dirac Operators}
\begin{abstract}
In this paper, we  prove two Kastler-Kalau-Walze type theorems for conformal perturbations of  twisted
 Dirac operators and conformal perturbations of twisted  signature operators on four-dimensional manifolds with (resp. without) boundary.
\end{abstract}
\begin{keyword}
Conformal perturbations of twisted Dirac operators; noncommutative residue;  non-unitary connection.
\end{keyword}
\end{frontmatter}
\section{Introduction}
\label{1}

The noncommutative residue found in \cite{Gu,Wo} plays a prominent role in noncommutative geometry.
 In \cite{Co1}, Connes used the noncommutative residue to derive a conformal four-
dimensional Polyakov action analogy. In \cite{Co2}, Connes proved that the noncommutative residue on a compact manifold
$M$ coincided with Dixmier's trace on pseudodifferential operators of order -dim$M$.
 Several years ago, Connes made a challenging observation that the noncommutative
residue of the square of the inverse of the Dirac operator was proportional to the Einstein-Hilbert action, which is
called the Kastler-Kalau-Walze theorem now. Kastler\cite{Ka} gave a brute-force proof of this theorem. Kalau and Walze\cite{KW} proved
this theorem in the normal coordinates system simultaneously. Ackermann\cite{Ac} gave a note on a new proof of this theorem
by means of the heat kernel expansion. Moreover, Fedosov etc.\cite{FGLS} constructed a noncommutative residue on the algebra of
classical elements in Boutet de Monvel's calculus on a compact manifold with boundary of dimension $n>2$. For Dirac operators
 and signature operators on manifolds with boundary, Wang\cite{Wa3} gave an operator-theoretic explanation of the gravitational
action for manifolds with boundary and proved a Kastler-Kalau-Walze type theorem.

In \cite{BZ}, Bismut and Zhang  introduced the de-Rham Hodge operator twisted by a flat vector bundle with a non-metric connection,
 and extended the famous Cheeger-M$\ddot{u}$ller theorem to the non-unitary case. In \cite{ZW}, Zhang considered the sub-signature
 operators twisted by a non-unitary flat vector  bundle and proved the associated Riemann-Roch theorem.
 In \cite{WW1}, we proved the Lichnerowicz formula for Dirac operators and signature operators twisted by a vector bundle with
 a non-unitary connection and got two Kastler-Kalau-Walze type theorems for twisted Dirac operators and twisted signature operators
 on four-dimensional manifolds with boundary. It is important that Wang proved a Kastler-Kalau-Walze type theorem for perturbations
 of Dirac operators on compact manifolds with or without boundary in \cite{Wa5}. The motivation of this paper is to establish two
  Kastler-Kalau-Walze type theorems for conformal perturbations of twisted Dirac operators and perturbations of twisted signature operators.

This paper is organized as follows. In Section 2, we prove a Kastler-Kalau-Walze type theorem for
conformal perturbations of twisted Dirac operators on 4-dimensional compact manifolds with or without boundary. In Section 3, we prove
a Kastler-Kalau-Walze type theorem for conformal perturbations of twisted signature operators on 4-dimensional compact
manifolds with or without boundary.

\section{A Kastler-Kalau-Walze Type Theorem for
 Conformal Perturbations of  twisted Dirac Operators}

 \subsection{Boutet de Monvel's calculus and noncommutative residue}
In this section, we recall some basic facts and formulae about Boutet de Monvel's calculus as follows.

Let $$ F:L^2({\bf R}_t)\rightarrow L^2({\bf R}_v);~F(u)(v)=\int e^{-ivt}u(t)\texttt{d}t$$ denote the Fourier transformation and
$\varphi(\overline{{\bf R}^+}) =r^+\varphi({\bf R})$ (similarly define $\varphi(\overline{{\bf R}^-}$)), where $\varphi({\bf R})$
denotes the Schwartz space and
  \begin{equation}
r^{+}:C^\infty ({\bf R})\rightarrow C^\infty (\overline{{\bf R}^+});~ f\rightarrow f|\overline{{\bf R}^+};~
 \overline{{\bf R}^+}=\{x\geq0;x\in {\bf R}\}.
\end{equation}
We define $H^+=F(\varphi(\overline{{\bf R}^+}));~ H^-_0=F(\varphi(\overline{{\bf R}^-}))$ which are orthogonal to each other. We have the following
 property: $h\in H^+~(H^-_0)$ iff $h\in C^\infty({\bf R})$ which has an analytic extension to the lower (upper) complex
half-plane $\{{\rm Im}\xi<0\}~(\{{\rm Im}\xi>0\})$ such that for all nonnegative integer $l$,
 \begin{equation}
\frac{\texttt{d}^{l}h}{\texttt{d}\xi^l}(\xi)\sim\sum^{\infty}_{k=1}\frac{\texttt{d}^l}{\texttt{d}\xi^l}(\frac{c_k}{\xi^k})
\end{equation}
as $|\xi|\rightarrow +\infty,{\rm Im}\xi\leq0~({\rm Im}\xi\geq0)$.

 Let $H'$ be the space of all polynomials and $H^-=H^-_0\bigoplus H';~H=H^+\bigoplus H^-.$ Denote by $\pi^+~(\pi^-)$ respectively the
 projection on $H^+~(H^-)$. For calculations, we take $H=\widetilde H=\{$rational functions having no poles on the real axis$\}$ ($\tilde{H}$
 is a dense set in the topology of $H$). Then on $\tilde{H}$,
 \begin{equation}
\pi^+h(\xi_0)=\frac{1}{2\pi i}\lim_{u\rightarrow 0^{-}}\int_{\Gamma^+}\frac{h(\xi)}{\xi_0+iu-\xi}\texttt{d}\xi,
\end{equation}
where $\Gamma^+$ is a Jordan close curve included ${\rm Im}\xi>0$ surrounding all the singularities of $h$ in the upper half-plane and
$\xi_0\in {\bf R}$. Similarly, define $\pi^{'}$ on $\tilde{H}$,
 \begin{equation}
\pi'h=\frac{1}{2\pi}\int_{\Gamma^+}h(\xi)\texttt{d}\xi.
\end{equation}
So, $\pi'(H^-)=0$. For $h\in H\bigcap L^1(R)$, $\pi'h=\frac{1}{2\pi}\int_{R}h(v)\texttt{d}v$ and for $h\in H^+\bigcap L^1(R)$, $\pi'h=0$.
Denote by $\mathcal{B}$ Boutet de Monvel's algebra (for details, see Section 2 of \cite{Wa1}).

An operator of order $m\in {\bf Z}$ and type $d$ is a matrix
$$A=\left(\begin{array}{lcr}
  \pi^+P+G  & K  \\
   T  &  S
\end{array}\right):
\begin{array}{cc}
\   C^{\infty}(X,E_1)\\
 \   \bigoplus\\
 \   C^{\infty}(\partial{X},F_1)
\end{array}
\longrightarrow
\begin{array}{cc}
\   C^{\infty}(X,E_2)\\
\   \bigoplus\\
 \   C^{\infty}(\partial{X},F_2)
\end{array}.
$$
where $X$ is a manifold with boundary $\partial X$ and
$E_1,E_2~(F_1,F_2)$ are vector bundles over $X~(\partial X
)$.~Here,~$P:C^{\infty}_0(\Omega,\overline {E_1})\rightarrow
C^{\infty}(\Omega,\overline {E_2})$ is a classical
pseudodifferential operator of order $m$ on $\Omega$, where
$\Omega$ is an open neighborhood of $X$ and
$\overline{E_i}|X=E_i~(i=1,2)$. $P$ has an extension:
$~{\cal{E'}}(\Omega,\overline {E_1})\rightarrow
{\cal{D'}}(\Omega,\overline {E_2})$, where
${\cal{E'}}(\Omega,\overline {E_1})~({\cal{D'}}(\Omega,\overline
{E_2}))$ is the dual space of $C^{\infty}(\Omega,\overline
{E_1})~(C^{\infty}_0(\Omega,\overline {E_2}))$. Let
$e^+:C^{\infty}(X,{E_1})\rightarrow{\cal{E'}}(\Omega,\overline
{E_1})$ denote extension by zero from $X$ to $\Omega$ and
$r^+:{\cal{D'}}(\Omega,\overline{E_2})\rightarrow
{\cal{D'}}(\Omega, {E_2})$ denote the restriction from $\Omega$ to
$X$, then define
$$\pi^+P=r^+Pe^+:C^{\infty}(X,{E_1})\rightarrow {\cal{D'}}(\Omega,
{E_2}).$$
In addition, $P$ is supposed to have the
transmission property; this means that, for all $j,k,\alpha$, the
homogeneous component $p_j$ of order $j$ in the asymptotic
expansion of the
symbol $p$ of $P$ in local coordinates near the boundary satisfies:
$$\partial^k_{x_n}\partial^\alpha_{\xi'}p_j(x',0,0,+1)=
(-1)^{j-|\alpha|}\partial^k_{x_n}\partial^\alpha_{\xi'}p_j(x',0,0,-1),$$
then $\pi^+P:C^{\infty}(X,{E_1})\rightarrow C^{\infty}(X,{E_2})$
by Section 2.1 of \cite{Wa1}.

In the following, write $\pi^+D^{-1}=\left(\begin{array}{lcr}
  \pi^+D^{-1}  & 0  \\
   0  &  0
\end{array}\right)$,
 we will compute  $\widetilde{Wres}[\pi^{+}(\widetilde{D}_{F}^{*})^{-1} \circ\pi^{+}\widetilde{D}_{F}^{-1}]$.
Let $M$ be a compact manifold with boundary $\partial M$. We assume that the metric $g^{M}$ on $M$ has
the following form near the boundary
 \begin{equation}
 g^{M}=\frac{1}{h(x_{n})}g^{\partial M}+\texttt{d}x _{n}^{2} ,
\end{equation}
where $g^{\partial M}$ is the metric on $\partial M$. Let $U\subset
M$ be a collar neighborhood of $\partial M$ which is diffeomorphic $\partial M\times [0,1)$. By the definition of $h(x_n)\in C^{\infty}([0,1))$
and $h(x_n)>0$, there exists $\tilde{h}\in C^{\infty}((-\varepsilon,1))$ such that $\tilde{h}|_{[0,1)}=h$ and $\tilde{h}>0$ for some
sufficiently small $\varepsilon>0$. Then there exists a metric $\hat{g}$ on $\hat{M}=M\bigcup_{\partial M}\partial M\times
(-\varepsilon,0]$ which has the form on $U\bigcup_{\partial M}\partial M\times (-\varepsilon,0 ]$
 \begin{equation}
\hat{g}=\frac{1}{\tilde{h}(x_{n})}g^{\partial M}+\texttt{d}x _{n}^{2} ,
\end{equation}
such that $\hat{g}|_{M}=g$.
We fix a metric $\hat{g}$ on the $\hat{M}$ such that $\hat{g}|_{M}=g$.
Note $\widetilde{D}_{F}$ is the  twisted Dirac operator on the spinor bundle $S(TM)\otimes F$ corresponding to the
connection $\widetilde{\nabla}$.

Now we recall the main theorem in \cite{FGLS}.

\begin{thm}\label{th:32}{\bf(Fedosov-Golse-Leichtnam-Schrohe)}
 Let $X$ and $\partial X$ be connected, ${\rm dim}X=n\geq3$,
 $A=\left(\begin{array}{lcr}\pi^+P+G &   K \\
T &  S    \end{array}\right)$ $\in \mathcal{B}$ , and denote by $p$, $b$ and $s$ the local symbols of $P,G$ and $S$ respectively.
 Define:
 \begin{eqnarray}
{\rm{\widetilde{Wres}}}(A)&=&\int_X\int_{\bf S}{\rm{tr}}_E\left[p_{-n}(x,\xi)\right]\sigma(\xi)dx \nonumber\\
&&+2\pi\int_ {\partial X}\int_{\bf S'}\left\{{\rm tr}_E\left[({\rm{tr}}b_{-n})(x',\xi')\right]+{\rm{tr}}
_F\left[s_{1-n}(x',\xi')\right]\right\}\sigma(\xi')dx',
\end{eqnarray}
Then~~ a) ${\rm \widetilde{Wres}}([A,B])=0 $, for any
$A,B\in\mathcal{B}$;~~ b) It is a unique continuous trace on
$\mathcal{B}/\mathcal{B}^{-\infty}$.
\end{thm}

 \subsection{A Kastler-Kalau-Walze Type Theorem for Conformal
Perturbations of  twisted Dirac Operators}

 In this section, we shall prove a Kastler-Kalau-Walze type formula for conformal
perturbations of  twisted Dirac Operators on four-dimensional compact manifolds
with boundary. Let $S(TM)$ be the spinors bundle and $F$ be an additional smooth vector bundle  equipped with a non-unitary
connection $\widetilde{\nabla}^{F}$.
Let $\widetilde{\nabla}^{F,\ast}$ be the dual connection  on $F$, and define
 \begin{equation}
\nabla^{F}=\frac{\widetilde{\nabla}^{F}+\widetilde{\nabla}^{F,\ast}}{2},~~\Phi=\frac{\widetilde{\nabla}^{F}-\widetilde{\nabla}^{F,\ast}}{2},
\end{equation}
then $\nabla^{F}$ is a metric connection and $\Phi$ is an endomorphism of $F$ with a 1-form coefficient.
 We consider the tensor product vector bundle $S(TM)\otimes F$,
which becomes a Clifford module via the definition:
\begin{equation}
c(a)=c(a)\otimes \texttt{id}_{F},~~~~~a\in TM,
\end{equation}
and which we equip with the compound connection:
 \begin{equation}
\widetilde{\nabla}^{ S(TM)\otimes F}= \nabla^{ S(TM)}\otimes \texttt{id}_{ F}+ \texttt{id}_{ S(TM)}\otimes \widetilde{\nabla}^{F}.
\end{equation}
The corresponding twisted Dirac operator $\widetilde{D}_{F}$ is locally specified as follows:
 \begin{equation}
\widetilde{D}_{F}=\sum_{i=1}^{n}c(e_{i})\widetilde{\nabla}^{ S(TM)\otimes F}_{e_{i}}.
\end{equation}
Let
 \begin{equation}
\nabla^{ S(TM)\otimes F}=\nabla^{ S(TM)}\otimes \texttt{id}_{ F}+ \texttt{id}_{ S(TM)}\otimes \nabla^{F},
\end{equation}
then the spinor connection  $\widetilde{\nabla}$  induced by $\nabla^{ S(TM)\otimes F}$ is locally
given by
 \begin{equation}
\widetilde{\nabla}^{ S(TM)\otimes F}=\nabla^{ S(TM)}\otimes \texttt{id}_{ F}+ \texttt{id}_{ S(TM)}\otimes \nabla^{F}+\texttt{id}_{ S(TM)}\otimes \Phi.
\end{equation}

\begin{defn}
Let  $\{e_{i}\}(1\leq i,j\leq n)$ $(\{\partial_{i}\})$ be the orthonormal
frames (natural frames respectively ) on  $TM$, then the form of Dirac operators as following
 \begin{equation}
D_{F}=\sum_{i,j}g^{ij}c(\partial_{i})\nabla^{ S(TM)\otimes F}_{\partial_{j}}=\sum_{j}^{n}c(e_{j})\nabla^{ S(TM)\otimes F}_{e_{j}},
\end{equation}
where $\nabla^{S(TM)\otimes F}_{\partial_{j}}=\partial_{j}+\sigma_{j}^{s}+\sigma_{j}^{F}$ and
$\sigma_{j}^{s}=\frac{1}{4}\sum_{j,k}\langle \nabla^{L}_{\partial_{i}}e_{j}, e_{k}\rangle c(e_{j})c(e_{k})$,
$\sigma_{j}^{F}$ is the connection  matrix of $\nabla^{F}$.
\end{defn}

 \begin{defn}
For sections $\psi\otimes \chi\in S(TM)\otimes F$, then the  twisted Dirac operators $\widetilde{D}_{F}$, $\widetilde{D}^{*}_{F}$ associated to the  connection
$\widetilde{\nabla}$ as follows
\begin{eqnarray}
&&\widetilde{D}_{F}(\psi\otimes \chi)=D_{F}(\psi\otimes \chi)+\sum_{i=1}^{n}c(e_{i})\otimes \Phi(e_{i})(\psi\otimes \chi),\\
&&\widetilde{D}^{*}_{F}(\psi\otimes \chi)=D_{F}(\psi\otimes \chi)-\sum_{i=1}^{n}c(e_{i})\otimes \Phi^{*}(e_{i})(\psi\otimes \chi).
\end{eqnarray}
Here $\Phi^{*}(e_{i})$ denotes the adjoint of $\Phi(e_{i})$.
\end{defn}

In the following, we will compute the more general case
$\widetilde{Wres}[\pi^{+}\big(f\widetilde{D}_{F}^{-1}\big) \circ\pi^{+}\big(f^{-1}(\widetilde{D}_{F}^{*})^{-1}\big)]$ for nonzero smooth functions $f, f^{-1}$.
 Denote by $\sigma_{l}(A)$ the $l$-order symbol of an operator A. An application of (3.5) and (3.6) in \cite{Wa1} shows that
\begin{equation}
\widetilde{Wres}[\pi^{+}\big(f\widetilde{D}_{F}^{-1}\big) \circ\pi^{+}\big(f^{-1}(\widetilde{D}_{F}^{*})^{-1}\big)]
=Wres[f\widetilde{D}_{F}^{-1} \circ f^{-1}(\widetilde{D}_{F}^{*})^{-1}]+\int_{\partial M}\Psi,
\end{equation}
where
 \begin{eqnarray}
\Psi&=&\int_{|\xi'|=1}\int_{-\infty}^{+\infty}\sum_{j,k=0}^{\infty}\sum \frac{(-i)^{|\alpha|+j+k+\ell}}{\alpha!(j+k+1)!}
{\bf{trace}}_{S(TM)\otimes F}\Big[\partial_{x_{n}}^{j}\partial_{\xi'}^{\alpha}\partial_{\xi_{n}}^{k}\sigma_{r}^{+}
(f\widetilde{D}_{F}^{-1})(x',0,\xi',\xi_{n})\nonumber\\
&&\times\partial_{x_{n}}^{\alpha}\partial_{\xi_{n}}^{j+1}\partial_{x_{n}}^{k}\sigma_{l}(f^{-1}(\widetilde{D}_{F}^{*})^{-1})(x',0,\xi',\xi_{n})\Big]
\texttt{d}\xi_{n}\sigma(\xi')\texttt{d}x' ,
\end{eqnarray}
and the sum is taken over $r-k+|\alpha|+\ell-j-1=-n,r\leq-1,\ell\leq-1$.

Note that
 \begin{eqnarray}
Wres\big[f\widetilde{D}_{F}^{-1} f^{-1}(\widetilde{D}_{F}^{*})^{-1}\big]
&=&Wres\big[(\widetilde{D}_{F}^{*}f \widetilde{D}_{F}f^{-1})^{-1} \big] \nonumber\\
&=&Wres\big[\big(\widetilde{D}_{F}^{*}\widetilde{D}_{F}-\widetilde{D}_{F}^{*}c(df)f^{-1}\big)^{-1 }\big].
\end{eqnarray}
In order to calculate the symbol of operators
 $\widetilde{D}_{F}^{*}\widetilde{D}_{F}-\widetilde{D}_{F}^{*}c(df)f^{-1}$,
we recall the basic notions of Laplace type operators in Section 1 of \cite{PBG}.
Let $V$ be a vector bundle on $M$. Any differential operator $P$ of Laplace type has locally the form
 \begin{equation}
P=-\big(g^{ij}\partial_{i}\partial_{j}+A^{i}\partial_{i}+B\big),
\end{equation}
where $\partial_{i}$  is a natural local frame on $TM$ ,
 and $(g^{ij})_{1\leq i,j\leq n}$ is the inverse matrix associated with the metric
matrix  $(g_{ij})_{1\leq i,j\leq n}$ on $M$,
 and $A^{i}$ and $B$ are smooth sections
of $\texttt{End}(V)$ on $M$ (Endomorphism). If $P$ is a Laplace type
operator of the form (2.20), then there is a unique
connection $\nabla$ on $V$ and a unique Endomorphism $E$ such that
 \begin{equation}
P=-\Big[g^{ij}(\nabla_{\partial_{i}}\nabla_{\partial_{j}}-
 \nabla_{\nabla^{L}_{\partial_{i}}\partial_{j}})+E\Big],
\end{equation}
where $\nabla^{L}$ denotes the Levi-Civita connection on $M$. Moreover
(with local frames of $T^{*}M$ and $V$), $\nabla_{\partial_{i}}=\partial_{i}+\omega_{i} $
and $E$ are related to $g^{ij}$, $A^{i}$, and $B$ through
 \begin{eqnarray}
&&\omega_{i}=\frac{1}{2}g_{ij}\big(A^{i}+g^{kl}\Gamma_{ kl}^{j} \texttt{Id}\big),\\
&&E=B-g^{ij}\big(\partial_{i}(\omega_{i})+\omega_{i}\omega_{j}-\omega_{k}\Gamma_{ ij}^{k} \big),
\end{eqnarray}
where $\Gamma_{ kl}^{j}$ is the  Christoffel coefficient of $\nabla^{L}$.

The next task then is to prove $\widetilde{D}^{*}_{F}\widetilde{D}_{F}$ has the Laplace type form.
Let $\partial^{j}=g^{ij}\partial_{i}, \sigma^{i}=g^{ij}\sigma_{j}, \Gamma^{k}=g^{ij}\Gamma_{ij}^{k}$.
From (6a) in \cite{Ka}, we have
 \begin{equation}
\widetilde{D}^{*}_{F}\widetilde{D}_{F}=D_{F}^{2}-c(\Phi^{*})D_{F}+D_{F}c(\Phi)-c(\Phi^{*})c(\Phi),
\end{equation}
and
\begin{eqnarray}
-c(\Phi^{*})D_{F}+D_{F}c(\Phi)&=&-\sum_{j}c(\Phi^{*})c(e_{j})\Big[e_{j}+ \sigma^{S(TM)\otimes F}_{j}\Big]
+\sum_{j}c(e_{j})\otimes c(\Phi)e_{j}\nonumber\\
 &&+\sum_{j}c(e_{j})\otimes e_{j}\big(c(\Phi)\big)
+\sum_{j}\Big[c(e_{j})\sigma_{j}^{S(TM)}\otimes c(\Phi)+c(e_{j})\otimes \sigma_{j}^{ F}c(\Phi) \Big].
\end{eqnarray}
Combining (2.24)-(2.25), we obtain the specification of $\widetilde{D}_{F}^{*}\widetilde{D}_{F}-\widetilde{D}_{F}^{*}c(df)f^{-1}$.
\begin{eqnarray}
\widetilde{D}_{F}^{*}\widetilde{D}_{F}-\widetilde{D}_{F}^{*}c(df)f^{-1}
&=&-g^{ij}\partial_{i}\partial_{j}-2\sigma^{j}_{S(TM)\otimes F}\partial_{j}-g^{ij}c(\partial_{i})c(df)f^{-1}\partial_{j}\nonumber\\
&&+\Gamma^{k}\partial_{k}-\sum_{j}\Big[c(\Phi^{*})c(e_{j})-c(e_{j})\otimes c(\Phi) \Big]e_{j}\nonumber\\
                &&-g^{ij}\Big[\partial_{i}(\sigma^{j}_{S(TM)\otimes F}) +\sigma^{i}_{S(TM)\otimes F}\sigma^{j}_{S(TM)\otimes F}
                  -\Gamma_{ij}^{k}\sigma_{S(TM)\otimes F}^{k}\Big]\nonumber\\
                &&-\sum_{j}\Big[c(\Phi^{*})c(e_{j}) \Big]\sigma^{S(TM)\otimes F}_{j}
               +\sum_{j}c(e_{j})\otimes e_{j}\big(c(\Phi)\big)\nonumber\\
                &&+\sum_{j}\Big[c(e_{j})\sigma_{j}^{S(TM)}\otimes c(\Phi)+c(e_{j})\otimes \sigma_{j}^{ F}c(\Phi) \Big]-c(\Phi^{*})c(\Phi)\nonumber\\
                               &&+\frac{1}{4}s+\frac{1}{2}\sum_{i\neq j} R^{F}(e_{i},e_{j})c(e_{i})c(e_{j})
                      -g^{ij}c(\partial_{i})\frac{\partial [c(df)f^{-1}]}{\partial x_{j}}\nonumber\\
                      &&-\sum_{j}g^{ij}\sigma^{j}_{S(TM)\otimes F} c(df)f^{-1}-c(\Phi^{*})c(df)f^{-1}.
\end{eqnarray}

In terms of local coordinates $\{\partial_{i}\}$ inducing the  coordinate transformation
 $e_{j}=\sum_{k=1}^{n}\langle e_{j}, \texttt{d}x^{k}\rangle \partial_{k}$, let $\Gamma^{k}=g^{ij}\Gamma_{ij}^{k} $,  then
 \begin{equation}
\omega_{j}=\sigma^{j}_{S(TM)}+\sigma^{j}_{F}+g^{ij}c(\partial_{i})c(df)f^{-1}+\frac{1}{2}\Big[\sum_{j,k=1}^{n}\langle e_{k}, \texttt{d}x^{j}\rangle c(\Phi^{*})c(e_{k})
                 -\sum_{j,k=1}^{n}\langle e_{k}, \texttt{d}x^{j}\rangle c(e_{k})c(\Phi)+\Gamma^{i}\Big].
\end{equation}
For a smooth vector field $X\in \Gamma(M,TM)$, let $c(X)$ denote the Clifford action. By direct computation in normal coordinates, we obtain
 \begin{equation}
\widetilde{\nabla}_{X}=\nabla^{ S(TM)\otimes F}_{X}+\frac{1}{2}[c(\Phi^{*})c(X)-c(X)c(\Phi)]+c(X)c(df)f^{-1}.
\end{equation}

 We now compute $E$.
Regrouping the terms and inserting (2.26), (2.27) into (2.23), we obtain
\begin{eqnarray}
E&=&g^{ij}\Big[\partial_{i}(\sigma^{j}_{S(TM)\otimes F}) +\sigma^{i}_{S(TM)\otimes F}\sigma^{j}_{S(TM)\otimes F}
                  -\Gamma_{ij}^{k}\sigma_{S(TM)\otimes F}^{k}\Big]
                  +\sum_{j}\Big[c(\Phi^{*})c(e_{j}) \Big]\sigma^{S(TM)\otimes F}_{j}\nonumber\\
                &&-\sum_{j}c(e_{j})\otimes e_{j}\big(c(\Phi)\big)
                -\sum_{j}\Big[c(e_{j})\sigma_{j}^{S(TM)}\otimes c(\Phi)+c(e_{j})\otimes \sigma_{j}^{ F}c(\Phi) \Big]\nonumber\\
                &&+c(\Phi^{*})c(\Phi)-\frac{1}{4}s-\frac{1}{2}\sum_{i\neq j} R^{F}(e_{i},e_{j})c(e_{i})c(e_{j})\nonumber\\
                 &&-\partial^{j}\big(\sigma^{j}_{S(TM)\otimes F}\big)-\frac{1}{2}\partial^{j}\Big(\sum_{k=1}^{n}\langle e_{k}, \texttt{d}x^{j}\rangle
                  c(\Phi^{*})c(e_{k})-\sum_{k=1}^{n}\langle e_{k}, \texttt{d}x^{j}\rangle c(e_{k})c(\Phi)\Big)\nonumber\\
                  &&-g^{ij}\sigma^{i}_{S(TM)\otimes F}\sigma^{j}_{S(TM)\otimes F}
                  -\frac{1}{2}g^{ij}\sigma^{i}_{S(TM)\otimes F}\Big(\sum_{k=1}^{n}\langle e_{k}, \texttt{d}x^{j}\rangle
                  c(\Phi^{*})c(e_{k})-\sum_{k=1}^{n}\langle e_{k}, \texttt{d}x^{j}\rangle c(e_{k})c(\Phi)\Big)\nonumber\\
                   &&-\frac{1}{2}g^{ij}\Big(\sum_{k=1}^{n}\langle e_{k}, \texttt{d}x^{i}\rangle
                  c(\Phi^{*})c(e_{k})-\sum_{k=1}^{n}\langle e_{k}, \texttt{d}x^{i}\rangle c(e_{k})c(\Phi)\Big)\sigma^{j}_{S(TM)\otimes F}\nonumber\\
                    &&-\frac{1}{4}g^{ij}\Big(\sum_{k=1}^{n}\langle e_{k}, \texttt{d}x^{i}\rangle
     c(\Phi^{*})c(e_{k})-\sum_{k=1}^{n}\langle e_{k}, \texttt{d}x^{i}\rangle c(e_{k})c(\Phi)\Big)\nonumber\\
  &&\times\Big(\sum_{k=1}^{n}\langle e_{k}, \texttt{d}x^{j}\rangle
                  c(\Phi^{*})c(e_{k})-\sum_{k=1}^{n}\langle e_{k}, \texttt{d}x^{j}\rangle c(e_{k})c(\Phi)\Big)\nonumber\\
                  &&+\Big[\sigma^{k}_{S(TM)\otimes F}+\frac{1}{2}\Big(\sum_{l=1}^{n}\langle e_{l}, \texttt{d}x^{k}\rangle c(\Phi^{*})c(e_{l})
                 -\sum_{l=1}^{n}\langle e_{l}, \texttt{d}x^{k}\rangle c(e_{l})c(\Phi)\Big)\Big]\Gamma^{k}\nonumber\\
          &&+g^{ij}c(\partial_{i})\sigma^{j}_{S(TM)\otimes F}c(df)f^{-1}+g^{ij}c(\partial_{i})\frac{\partial [c(df)f^{-1}]}{\partial x_{j}}
          -c(\Phi^{*})c(df)f^{-1} \nonumber\\
          &&- g^{ij}\partial_{j}(\frac{1}{2}c(\partial_{i})c(df)f^{-1})-\frac{1}{4}g^{ij}c(\partial_{i})c(df)f^{-1}c(\partial_{i})c(df)f^{-1}
          +\frac{1}{2}g^{ij}c(\partial_{k})c(df)f^{-1}\Gamma_{ij}^{k}\nonumber\\
          &&-g^{ij}\sigma^{j}_{S(TM)\otimes F}c(\partial_{i}) c(df)f^{-1}
          -\Big[c(\Phi^{*})c(e_{i})-c(e_{i})c(\Phi)\Big]c(\partial_{i}) c(df)f^{-1}.
\end{eqnarray}

Since $E$ is globally defined on $M$, so we can perform
computations of $E$ in normal coordinates. In terms of normal coordinates about $x_{0}$ one has:
$\sigma^{j}_{S(TM)}(x_{0})=0$ $e_{j}\big(c(e_{i})\big)(x_{0})=0$, $\Gamma^{k}(x_{0})=0$, we conclude that
\begin{eqnarray}
E(x_{0})&=&-\frac{1}{4}s-\frac{1}{2}\sum_{i\neq j} R^{F}(e_{i},e_{j})c(e_{i})c(e_{j})
         -\frac{1}{4}\sum_{i}\Big[c(\Phi^{*})c(e_{i})-c(e_{i})c(\Phi) \Big]^{2}+c(\Phi^{*})c(\Phi) \nonumber\\
        && -\frac{1}{2}\sum_{j}\Big(\nabla_{e_{j}}^{F}c(\Phi^{*})\Big)c(e_{j})-\frac{1}{2}\sum_{j}c(e_{j})\nabla_{e_{j}}^{F}c(\Phi)\nonumber\\
        &&+c(\partial_{i})\frac{\partial (c(df)f^{-1})}{\partial x_{i}}
          -c(\Phi^{*})c(df)f^{-1} \nonumber\\
          &&- \partial_{i}\big(\frac{1}{2}c(\partial_{i})c(df)f^{-1}\big)-\frac{1}{4}c(\partial_{i})c(df)f^{-1}c(\partial_{i})c(df)f^{-1}
          \nonumber\\
          && -\Big[c(\Phi^{*})c(e_{i})-c(e_{i})c(\Phi)\Big]c(\partial_{i}) c(df)f^{-1}.
\end{eqnarray}

From  Theorem 1 in \cite{Ka} and Theorem 1 in \cite{KW},
for $M$ a compact $n$ dimensional ($n\geq 4$, even) Riemannian manifold and $\widetilde{D}^{*}_{F}\widetilde{D}_{F}$
a generalized Laplacian acting on sections of  vector bundle  on $M$, the
following relation holds:
\begin{equation}
Wres\big[\big(\widetilde{D}_{F}^{*}\widetilde{D}_{F}-\widetilde{D}_{F}^{*}c(df)f^{-1}\big)^{-1} \big] ^{(\frac{n-2}{2})}
=\frac{(2\pi)^{\frac{n}{2}}}{(\frac{n}{2}-2)!}\int_{M}{\bf{Tr}}(\frac{s}{6}+E)\texttt{d}vol_{M},
\end{equation}
where Wres denotes the noncommutative residue.

\begin{lem}The following identity holds
\begin{eqnarray*}
&&{\bf{Tr}}[c(\Phi^{*})c(df)]=-{\bf{Tr}}_{F}[\Phi^{*}(grad_{M}f)]{\bf{Tr}}[id];\nonumber\\
&&{\bf{Tr}}[c(\partial_{i})\frac{\partial (c(df)f^{-1})}{\partial x_{i}}](x_{0})=[-f^{-1} \Delta(f)-\langle grad_{M}f, grad_{M}f^{-1}\rangle](x_{0}){\bf{Tr}}[id];\nonumber\\
&&{\bf{Tr}}[\partial_{i}\big(c(\partial_{i})c(df)f^{-1}\big)](x_{0})=[- f^{-1} \Delta(f)(x_{0})
-\langle grad_{M}f, grad_{M}f^{-1}\rangle](x_{0}){\bf{Tr}}[id];\nonumber\\
&&{\bf{Tr}}[c(\partial_{i})c(df)f^{-1}c(\partial_{i})c(df)f^{-1}](x_{0})=f^{-2}\big[ |grad_{M}(f)|^{2}+2\Delta(f)\big](x_{0}){\bf{Tr}}[id];\nonumber\\
&&{\bf{Tr}}[c(\Phi^{*})c(e_{i})c(\partial_{i})c(df)f^{-1}](x_{0})=-f^{-1} {\bf{Tr}}_{F}[\Phi^{*}(grad_{M}f)](x_{0}){\bf{Tr}}[id];\nonumber\\
&&{\bf{Tr}}[c(e_{i})c(\Phi)c(\partial_{i}) c(df)f^{-1}](x_{0})=f^{-1} {\bf{Tr}}_{F}[\Phi (grad_{M}f)](x_{0}){\bf{Tr}}[id].
\end{eqnarray*}
\end{lem}
\begin{proof}
By the relation of the Clifford action and $\texttt{tr}AB=\texttt{tr}BA$,
\begin{eqnarray*}
&&{\bf{Tr}}[c(\xi')c(\texttt{d}x_{n})]=0; \ {\bf{Tr}}[c(\texttt{d}x_{n})^{2}]=-4;\ {\bf{Tr}}[c(\xi')^{2}](x_{0})|_{|\xi'|=1}=-4;\nonumber\\
&&{\bf{Tr}}[\partial_{x_{n}}[c(\xi')]c(\texttt{d}x_{n})]=0; \ {\bf{Tr}}[\partial_{x_{n}}c(\xi')\times c(\xi')](x_{0})|_{|\xi'|=1}=-2h'(0).
\end{eqnarray*}
Let $c(\partial_{i})= \sum_{j=1}^{4} \langle \partial_{i}, \widetilde{e}_{k}\rangle c(\widetilde{e}_{k})$ and
$c(\Phi^{*})= \sum_{j=1}^{4}c(e_{j})\otimes\Phi^{*}(e_{j})$, then
\begin{eqnarray*}
{\bf{Tr}}[c(\Phi^{*})c(df)](x_{0})&=&{\bf{Tr}}[\sum_{j=1}^{4}c(e_{j})\otimes\Phi^{*}(e_{j})c(df)](x_{0})
 =\sum_{j=1}^{4}{\bf{Tr}}[c(e_{j})c(df)]{\bf{Tr}}_{F}[\Phi^{*}(e_{j})](x_{0})\nonumber\\
&=&-\sum_{j=1}^{4}g(e_{j}, grad_{M}f){\bf{Tr}}[id]{\bf{Tr}}_{F}[\Phi^{*}(e_{j})](x_{0})
 =-\sum_{j=1}^{4}e_{j}(f){\bf{Tr}}[id]{\bf{Tr}}_{F}[\Phi^{*}(e_{j})](x_{0})\nonumber\\
&=&-{\bf{Tr}}[id]{\bf{Tr}}_{F}[\Phi^{*}(\sum_{j=1}^{4}e_{j}(f)e_{j})](x_{0})=-{\bf{Tr}}_{F}[\Phi^{*}(grad_{M}f)](x_{0}){\bf{Tr}}[id]\nonumber\\
&=&-{\bf{Tr}}_{F}[\Phi^{*}(grad_{M}f)](x_{0}){\bf{Tr}}[id],
\end{eqnarray*}
and
\begin{eqnarray*}
&&{\bf{Tr}}[\partial_{i}\big(c(\partial_{i})c(df)f^{-1}\big)](x_{0})\nonumber\\
&=&{\bf{Tr}}[\partial_{i}\big(c(\partial_{i})\big)c(df)f^{-1}](x_{0})+{\bf{Tr}}[c(\partial_{i})\partial_{i}\big(c(df)\big)f^{-1}](x_{0})
    +{\bf{Tr}}[c(\partial_{i})c(df)\partial_{i}\big(f^{-1}\big)](x_{0})\nonumber\\
&=&{\bf{Tr}}[\partial_{i}\big(\sum_{k=1}^{4} \langle \partial_{i}, \widetilde{e}_{k}\rangle c(\widetilde{e}_{k})\big)c(df)f^{-1}](x_{0})
            +f^{-1}{\bf{Tr}}[c(\partial_{i})\partial_{i}\big(c(df)\big)](x_{0})\nonumber\\
      &&-g_{TM}(\partial_{i},grad_{M}f) \partial_{i}(f^{-1})(x_{0}){\bf{Tr}}[id]
    \nonumber\\
&=&\sum_{k=1}^{4}{\bf{Tr}}[\partial_{i}\big(g_{TM}(\partial_{i}, \widetilde{e}_{k})\big) c(\widetilde{e}_{k})c(df)f^{-1}](x_{0})
         +f^{-1}{\bf{Tr}}[c(\partial_{i})\sum_{j=1}^{4}\frac{\partial\big(\tilde{e}_{j}(f)\big)}{\partial x_{i}} c(\tilde{e}_{j}) \big)](x_{0})\nonumber\\
          && -\partial_{i}(f)\partial_{j}(f^{-1})(x_{0}){\bf{Tr}}[id]
      \nonumber\\
&=&\sum_{k=1}^{4}\partial_{i}\big(g_{TM}(\partial_{i}, \widetilde{e}_{k})\big) f^{-1}{\bf{Tr}}[ c(\widetilde{e}_{k})c(grad_{M}f)](x_{0})
             +f^{-1}\sum_{j=1}^{4}\frac{\partial\big(\tilde{e}_{j}(f)}{\partial x_{i}} {\bf{Tr}}[c(\partial_{i})c(\tilde{e}_{j})\big)](x_{0})\nonumber\\
            && -\langle grad_{M}f, grad_{M}f^{-1}\rangle (x_{0}){\bf{Tr}}[id]\nonumber\\
&=&-\sum_{k=1}^{4}\partial_{i}\big(g_{TM}(\partial_{i}, \widetilde{e}_{k})\big) f^{-1}\widetilde{e}_{k}(f)(x_{0}){\bf{Tr}}[id]
    -f^{-1}\sum_{j=1}^{4}\frac{\partial\big(\tilde{e}_{j}(f)\big)}{\partial x_{i}}(x_{0}) {\bf{Tr}}[id]\nonumber\\
  &&-\langle grad_{M}f, grad_{M}f^{-1}\rangle (x_{0}){\bf{Tr}}[id]\nonumber\\
  &=&[- f^{-1}\sum_{i,k}\partial_{i}(g_{TM}(\partial_{i},\tilde{e}_{k}))\tilde{e}_{k}(f)(x_{0})- f^{-1} \Delta(f)(x_{0})
-\langle grad_{M}f, grad_{M}f^{-1}\rangle](x_{0}){\bf{Tr}}[id]\nonumber\\
  &=&[- f^{-1} \Delta(f)(x_{0})
-\langle grad_{M}f, grad_{M}f^{-1}\rangle](x_{0}){\bf{Tr}}[id].
\end{eqnarray*}
And similarly we have proved this lemma. For more trace expansions, we can see \cite{Co1,Co2,Wa3}.
\end{proof}
From (2.30), Lemma 2.4 and $Tr(e_{i}e_{j}) =0~(i\neq j)$, we find for the trace
\begin{eqnarray}
{\bf{Tr}}(E(x_{0}))&=&{\bf{Tr}}\Big[-\frac{1}{4}s
         +c(\Phi^{*})c(\Phi)-\frac{1}{4}\sum_{i}\big[c(\Phi^{*})c(e_{i})-c(e_{i})c(\Phi) \big]^{2}\nonumber\\
         &&~~-\frac{1}{2}\sum_{j}\nabla_{e_{j}}^{F}\big(c(\Phi^{*})\big)c(e_{j})-\frac{1}{2}\sum_{j}c(e_{j})\nabla_{e_{j}}^{F}
            \big(c(\Phi)\big)\Big]\nonumber\\
         &&~~-2f^{-1}\Delta(f)+4f^{-1}{\bf{Tr}}_{F}[\Phi(grad_{M}f)]-f^{-2}\big[ |grad_{M}(f)|^{2}+2\Delta(f)\big],
\end{eqnarray}
where  $\triangle$ denotes the Laplacian operator.

Substituting (2.32) into (2.31), we obtain
\begin{thm}
For even $n$-dimensional compact spin manifolds
without boundary, the following equality
holds:
\begin{eqnarray}
&&Wres\big[f\widetilde{D}_{F}^{-1} f^{-1}(\widetilde{D}_{F}^{*})^{-1}\big]^{\frac{n-2}{2}}\nonumber\\
&=&\frac{(2\pi)^{\frac{n}{2}}}{(\frac{n}{2}-2)!}\int_{M}
\Big\{{\bf{Tr}} \Big[-\frac{s}{12}+c(\Phi^{*})c(\Phi)-\frac{1}{4}\sum_{i}\big[c(\Phi^{*})c(e_{i})-c(e_{i})c(\Phi) \big]^{2}\nonumber\\
         &&~~~~-\frac{1}{2}\sum_{j}\nabla_{e_{j}}^{F}\big(c(\Phi^{*})\big)c(e_{j})-\frac{1}{2}\sum_{j}c(e_{j})\nabla_{e_{j}}^{F}
            \big(c(\Phi)\big)\Big]\nonumber\\
 &&~~-2f^{-1}\Delta(f)+4f^{-1}{\bf{Tr}}_{F}[\Phi(grad_{M}f)]-f^{-2}\big[ |grad_{M}(f)|^{2}+2\Delta(f)\big]\Big\}\texttt{d}vol_{M},
\end{eqnarray}
where $s$ is the scaler curvature.
\end{thm}

Locally we can use Theorem 2.5 to compute the interior term of (2.17), then  for conformal
perturbations of  twisted Dirac Operators on four-dimensional compact manifolds
with boundary,
 \begin{eqnarray}
&&\int_M\int_{|\xi|=1}{\rm trace}_{S(TM)\otimes F}[\sigma_{-4}((\widetilde{D}_{F}^{*}\widetilde{D}_{F}-\widetilde{D}_{F}^{*}c(df)f^{-1})^{-1}]\sigma(\xi)dx\nonumber\\
 &=& 4 \pi^{2}\int_{M}\Big\{{\bf{Tr}}
 \Big[-\frac{s}{12}+c(\Phi^{*})c(\Phi)
-\frac{1}{4}\sum_{i}\big[c(\Phi^{*})c(e_{i})-c(e_{i})c(\Phi) \big]^{2}
     \nonumber\\
     &&~~~~~~~~~~~~~~-\frac{1}{2}\sum_{j}\nabla_{e_{j}}^{F}\big(c(\Phi^{*})\big)c(e_{j})-\frac{1}{2}\sum_{j}c(e_{j})\nabla_{e_{j}}^{F}
            \big(c(\Phi)\big)\Big]
            \nonumber\\
 &&~~-2f^{-1}\Delta(f)+4f^{-1}{\bf{Tr}}_{F}[\Phi(grad_{M}f)]-f^{-2}\big[ |grad_{M}(f)|^{2}+2\Delta(f)\big]\Big\}\texttt{d}vol_{M}.
          \end{eqnarray}
So we only need to compute $\int_{\partial M}\Psi$.
Let us now turn to compute the symbol expansion of $\widetilde{D}_{F}^{-1}$. Recall the definition of the twisted  Dirac operator
 $\widetilde{D}_{F}$ in Definition 2.3.
Let $\nabla^{TM}$
denote the Levi-civita connection about $g^M$. In the local coordinates $\{x_i; 1\leq i\leq n\}$ and the fixed orthonormal frame
$\{\widetilde{e_1},\cdots,\widetilde{e_n}\}$, the connection matrix $(\omega_{s,t})$ is defined by
\begin{equation}
\nabla^{TM}(\widetilde{e_1},\cdots,\widetilde{e_n})= (\widetilde{e_1},\cdots,\widetilde{e_n})(\omega_{s,t}).
\end{equation}
Let $c(\widetilde{e_i})$ denote the Clifford action. Let $g^{ij}=g(dx_i,dx_j)$ and
\begin{equation}
\nabla^{TM}_{\partial_i}\partial_j=\sum_k\Gamma_{ij}^k\partial_k; ~\Gamma^k=g^{ij}\Gamma_{ij}^k.
\end{equation}
Let the cotangent vector $\xi=\sum \xi_jdx_j$ and $\xi^j=g^{ij}\xi_i$.
 By Lemma 1 in \cite{Wa1} and Lemma 2.1 in \cite{Wa3}, for any fixed point $x_0\in\partial M$, we can choose the normal coordinates $U$
 of $x_0$ in $\partial M$ (not in $M$). By the composition formula and (2.2.11) in \cite{Wa3}, we obtain
\begin{lem}
 Let $\widetilde{D}^{*}_{F},  \widetilde{D}_{F}$ be the twisted  Dirac operators on  $\Gamma(S(TM)\otimes F)$,
then
\begin{eqnarray}
&&\sigma_{-1}(( \widetilde{D}^{*}_{F} )^{-1})=\sigma_{-1}(\widetilde{D}_{F}^{-1})=\frac{\sqrt{-1}c(\xi)}{|\xi|^2}; \\
&&\sigma_{-2}((\widetilde{D}^{*}_{F})^{-1})=\frac{c(\xi)\sigma_0(\widetilde{D}^{*}_{F})c(\xi)}{|\xi|^4}+\frac{c(\xi)}{|\xi|^6}\sum_jc(dx_j)
\Big[\partial_{x_j}[c(\xi)]|\xi|^2-c(\xi)\partial_{x_j}(|\xi|^2)\Big] ;\\
&& \sigma_{-2}(\widetilde{D}_{F}^{-1})=\frac{c(\xi)\sigma_0(\widetilde{D}_{F})c(\xi)}{|\xi|^4}+\frac{c(\xi)}{|\xi|^6}\sum_jc(dx_j)
\Big[\partial_{x_j}[c(\xi)]|\xi|^2-c(\xi)\partial_{x_j}(|\xi|^2)\Big],
\end{eqnarray}
where
\begin{eqnarray}
\sigma_0(\widetilde{D}^{*}_{F})&=& \sigma_{0}(D)+\sum_{j=1}^{n}c(e_{j})\big(\sigma_{j}^{F}-\Phi^{*}(e_{j})\big);\\
\sigma_0(\widetilde{D}_{F})&=& \sigma_{0}(D)+\sum_{j=1}^{n}c(e_{j})\big(\sigma_{j}^{F}+\Phi(e_{j})\big).
\end{eqnarray}
\end{lem}

Let us now turn to compute $\Psi$ (see formula (2.18) for definition of $\Psi$). Since the sum is taken over $-r-\ell+k+j+|\alpha|=3,
 \ r, \ell\leq-1$, then we have the boundary term of (2.17) is the sum of  the following five terms.

 {\bf case a)~I)}~$r=-1,~l=-1~k=j=0,~|\alpha|=1$

From (2.18) we have
 \begin{eqnarray}
&&{\rm case~a)~I)}\nonumber\\
&=&-\int_{|\xi'|=1}\int^{+\infty}_{-\infty}\sum_{|\alpha|=1}
{\rm trace} \Big[\partial^\alpha_{\xi'}\pi^+_{\xi_n}\sigma_{-1}(\widetilde{D}_{F}^{-1})\times
\partial^\alpha_{x'}\partial_{\xi_n}\sigma_{-1}((\widetilde{D}^{*}_{F})^{-1})\Big](x_0)\texttt{d}\xi_n\sigma(\xi')\texttt{d}x'\nonumber\\
&&-f\sum_{j<n}\partial_{j}(f^{-1})\int_{|\xi'|=1}\int^{+\infty}_{-\infty}\sum_{|\alpha|=1}
{\rm trace} \Big[\partial^\alpha_{\xi'}\pi^+_{\xi_n}\sigma_{-1}(\widetilde{D}_{F}^{-1})\times
\partial_{\xi_n}\sigma_{-1}((\widetilde{D}^{*}_{F})^{-1})\Big](x_0)\texttt{d}\xi_n\sigma(\xi')\texttt{d}x'\nonumber\\
&=&0.
\end{eqnarray}
And similarly we get

{\bf case a)~II)}~$r=-1,~l=-1~k=|\alpha|=0,~j=1$
 \begin{eqnarray}
&&{\rm case \ a)~II)}\nonumber\\
&=&-\frac{1}{2}\int_{|\xi'|=1}\int^{+\infty}_{-\infty} {\rm
trace} \Big[\partial_{x_n}\pi^+_{\xi_n}\sigma_{-1}( \widetilde{D} _{F} ^{-1})\times
\partial_{\xi_n}^2\sigma_{-1}((\widetilde{D}^{*}_{F})^{-1})\Big](x_0)\texttt{d}\xi_n\sigma(\xi')\texttt{d}x'\nonumber\\
&&-f^{-1}\partial_{x_n}(f)\frac{1}{2}\int_{|\xi'|=1}\int^{+\infty}_{-\infty} {\rm
trace} \Big[\partial_{x_n} \sigma_{-1}( \widetilde{D} _{F} ^{-1})\times
\partial_{\xi_n}^2\sigma_{-1}((\widetilde{D}^{*}_{F})^{-1})\Big](x_0)\texttt{d}\xi_n\sigma(\xi')\texttt{d}x'\nonumber\\
&=&-\frac{3}{8}\pi h'(0) \texttt{dim}F\Omega_3\texttt{d}x'-\frac{\pi i}{2}\Omega_{3}f^{-1}\partial_{x_{n}}(f)\texttt{d}x'.
\end{eqnarray}

 {\bf case a)~III)}~$r=-1,~l=-1~j=|\alpha|=0,~k=1$
 \begin{eqnarray}
 &&{\rm case~ a)~III)}\nonumber\\
&=&-\frac{1}{2}\int_{|\xi'|=1}\int^{+\infty}_{-\infty}
{\rm trace} \Big[\partial_{\xi_n}\pi^+_{\xi_n}\sigma_{-1}(\widetilde{D}_{F}^{-1})\times
\partial_{\xi_n}\partial_{x_n}\sigma_{-1}((\widetilde{D}^{*}_{F})^{-1})\Big](x_0)\texttt{d}\xi_n\sigma(\xi')\texttt{d}x'\nonumber\\
&&-f\partial_{x_{n}}(f^{-1})\Big[\partial_{\xi_n}\pi^+_{\xi_n}\sigma_{-1}(\widetilde{D}_{F}^{-1})\times
\partial_{\xi_n}\partial_{x_n}\sigma_{-1}((\widetilde{D}^{*}_{F})^{-1})\Big](x_0)\texttt{d}\xi_n\sigma(\xi')\texttt{d}x'\nonumber\\
&=&\frac{3}{8}\pi h'(0) \texttt{dim}F\Omega_3\texttt{d}x'+\frac{\pi i}{2}\Omega_{3}f\partial_{x_{n}}(f^{-1})\texttt{d}x'.
\end{eqnarray}

 {\bf case b)}~$r=-2,~l=-1,~k=j=|\alpha|=0$
\begin{eqnarray}
{\rm case~ b)}&=&-i\int_{|\xi'|=1}\int^{+\infty}_{-\infty}
{\rm trace} \Big[\pi^+_{\xi_n}\sigma_{-2}(f\widetilde{D}_{F}^{-1})\times
\partial_{\xi_n}\sigma_{-1}(f^{-1}(\widetilde{D}^{*}_{F})^{-1})\Big](x_0)\texttt{d}\xi_n\sigma(\xi')\texttt{d}x' \nonumber\\
&=&\Big[\frac{9}{8} h'(0) \texttt{dim}F
-\frac{1}{4}\texttt{Tr}\Big(\texttt{id}\otimes \big(\sigma_{n}^{F}-\Phi^{*}(e_{j})\big)\Big)\Big]\pi \Omega_{3}\texttt{d}x'.
\end{eqnarray}

{\bf  case c)}~$r=-1,~l=-2,~k=j=|\alpha|=0$
\begin{eqnarray}
{\rm case~ c)}&=&-i\int_{|\xi'|=1}\int^{+\infty}_{-\infty}{\rm trace}\Big [\pi^+_{\xi_n}\sigma_{-1}(f\widetilde{D}_{F}^{-1})\times
\partial_{\xi_n}\sigma_{-2}(f^{-1}(\widetilde{D}^{*}_{F})^{-1})\Big](x_0)d\xi_n\sigma(\xi')dx'\nonumber\\
&=&\Big[-\frac{9}{8} h'(0) \texttt{dim}F+\frac{1}{4}\texttt{Tr}[\texttt{id}\otimes (\sigma_{n}^{F}+\Phi(e_{n}))]\Big]
\pi \Omega_{3}\texttt{d}x'.
\end{eqnarray}
We note that $\texttt{dim} S(TM)=4$, now $\Psi$  is the sum of the \textbf{case (a, b, c)}, so
\begin{equation}
\sum \textbf{case a, b , c}=\frac{\pi i}{2}\Omega_{3} \big[f\partial_{x_{n}}(f^{-1})- f^{-1}\partial_{x_{n}}(f)\big]\texttt{d}x'+
\texttt{Tr}_{F}\big(\Phi^{*}(e_{n})+\Phi(e_{n})\big)\pi \Omega_{3}\texttt{d}x'.
\end{equation}
Hence we conclude that
\begin{thm}
 Let M be a 4-dimensional compact manifolds   with the boundary $\partial M$, for conformal perturbations of  twisted Dirac operators $\widetilde{D}_{F}$, then
 \begin{eqnarray}
&&\widetilde{Wres}[\pi^{+}\big(f\widetilde{D}_{F}^{-1}\big) \circ\pi^{+}\big(f^{-1}(\widetilde{D}_{F}^{*})^{-1}\big)]\nonumber\\
&=&4 \pi^{2}\int_{M}\Big\{{\bf{Tr}}
 \Big[-\frac{s}{12}+c(\Phi^{*})c(\Phi)
-\frac{1}{4}\sum_{i}\big[c(\Phi^{*})c(e_{i})-c(e_{i})c(\Phi) \big]^{2}
     \nonumber\\
     &&-\frac{1}{2}\sum_{j}\nabla_{e_{j}}^{F}\big(c(\Phi^{*})\big)c(e_{j})-\frac{1}{2}\sum_{j}c(e_{j})\nabla_{e_{j}}^{F}
            \big(c(\Phi)\big)\Big]
\nonumber\\
&&~~-2f^{-1}\Delta(f)+4f^{-1}{\bf{Tr}}_{F}[\Phi(grad_{M}f)]-f^{-2}\big[ |grad_{M}(f)|^{2}+2\Delta(f)\big]\Big\}\texttt{d}vol_{M} \nonumber\\
 &&~~~~+\int_{\partial_{ M}}\frac{\pi i}{2}\Omega_{3} \big[f\partial_{x_{n}}(f^{-1})- f^{-1}\partial_{x_{n}}(f)\big]\texttt{d}x'+
{\bf{Tr}}_{F}\big(\Phi^{*}(e_{n})+\Phi(e_{n})\big)\pi \Omega_{3}\texttt{d}x',
\end{eqnarray}
where  $s$ is the scalar curvature.
\end{thm}

\section{A Kastler-Kalau-Walze Type Theorem for Conformal Perturbations of  twisted signature Operators}
Let us recall the definition of twisted signature operators. We consider a $n$-dimensional oriented Riemannian manifold $(M, g^{M})$.
Let $F$ be a real vector bundle over $M$. let $g^{F}$ be an Euclidean metric on $F$. Let
 \begin{equation}
\wedge^{\ast}(T^{\ast}M)=\bigoplus_{i=0}^{n}\wedge^{i}(T^{\ast}M)
\end{equation}
be the real exterior algebra bundle of $T^{\ast}M$. Let
 \begin{equation}
\Omega^{\ast}(M,F)=\bigoplus_{i=0}^{n}\Omega^{i}(M,F)=\bigoplus_{i=0}^{n}C^{\infty}(M,\wedge^{\ast}(T^{\ast}M)\otimes F)
\end{equation}
be the set of smooth sections of $\wedge^{\ast}(T^{\ast}M)\otimes F$. Let $\ast$ be the Hodge star operator of $g^{TM}$.
It extends  on  $\wedge^{\ast}(T^{\ast}M)\otimes F$ by acting on $F$ as identity. Then $\Omega^{\ast}(M,F)$ inherits the following
standardly induced inner product
 \begin{equation}
\langle \alpha, \beta  \rangle=\int_{M}\langle \alpha\wedge^{\ast}\beta  \rangle_{F},~~~~\alpha, \beta \in\Omega^{\ast}(M,F).
\end{equation}
Denote by $\widehat{\nabla}^{F}$ the non-Euclidean connection on $F$. Let $d^{F}$ be the obvious extension of $\nabla^{F}$ on $\Omega^{\ast}(M,F)$.
 and $\delta^{F}=d^{F\ast}$ be the formal adjoint operator of $d^{F}$ with respect to the inner product.
 Then the differential
operator   $\hat{D}^{F}$  acting on $\Omega^{\ast}(M,F)$  can be defined by
  \begin{equation}
\hat{D}^{F}=d^{F}+\delta^{F}.
\end{equation}
Let
 \begin{equation}
\omega(F,g^{F})=\widehat{\nabla}^{F,\ast}-\widehat{\nabla}^{F},~~\nabla^{F,e}=\nabla^{F}+\frac{1}{2}\omega(F,g^{F}).
\end{equation}
Then $\nabla^{F,e}$ is an Euclidean connection on $(F,g^{F})$.

Let $\nabla^{\wedge^{\ast}(T^{\ast}M)}$ be the Euclidean connection on $\wedge^{\ast}(T^{\ast}M)$ induced canonically by the Levi-Civita
connection $\nabla^{TM}$ of $g^{TM}$. Let $\nabla^{e}$ be the Euclidean connection on $\wedge^{\ast}(T^{\ast}M)\otimes F$ obtained from
 the tensor product of $\nabla^{\wedge^{\ast}(T^{\ast}M)}$ and $\nabla^{F,e}$.
Let $\{e_{1},\cdots,e_{n}\}$ be an oriented (local) orthonormal basis of $TM$. The following result was proved by Proposition in \cite{BZ}.

 \begin{prop} \cite{BZ}
The following identity holds
 \begin{equation}
d^{F}+\delta^{F}=\sum_{i=1}^{n}c(e_{i})\nabla^{e}_{e_{i}}-\frac{1}{2}\sum_{i=1}^{n}\hat{c}(e_{i})\omega(F,g^{F})(e_{i}).
\end{equation}
\end{prop}
Let
 \begin{equation}
D_{F}^{e}=\sum_{j=1}^{n}c(e_{j})\nabla^{e}_{e_{j}},
\end{equation}
then the  twisted signature operators $\hat{D}_{F}$, $\hat{D}^{*}_{F}$ as follows.
 \begin{defn}\label{1}
For sections $\psi\otimes \chi\in \wedge^{\ast}(T^{\ast}M)\otimes F$,
\begin{eqnarray}\label{eq:1}
&&\hat{D}_{F}(\psi\otimes \chi)=D_{F}^{e}(\psi\otimes \chi)
 -\frac{1}{2}\sum_{i=1}^{n}\hat{c}(e_{i})\omega(F,g^{F})(e_{i})(\psi\otimes \chi),\\
 \label{eq:2}
&&\hat{D}^{*}_{F}(\psi\otimes \chi)=D_{F}^{*,e}(\psi\otimes \chi)
-\frac{1}{2}\sum_{i=1}^{n}\hat{c}(e_{i})\omega^{*}(F,g^{F})(e_{i})(\psi\otimes \chi).
\end{eqnarray}
Here $\omega^{*}(F,g^{F})(e_{i})$ denotes the adjoint of $\omega(F,g^{F})(e_{i})$.
\end{defn}

In the following, we will compute the more general case
$\widetilde{Wres}\big[\pi^{+}\big(f\hat{D}_{F}^{-1}\big) \circ \pi^{+}\big(f^{-1}(\hat{D}_{F}^{*})^{-1}\big)\big]$ for nonzero smooth functions $f, f^{-1}$.
 Denote by $\sigma_{l}(A)$ the $l$-order symbol of an operator A. An application of (3.5) and (3.6) in \cite{Wa1} shows that
  \begin{eqnarray}
\widetilde{Wres}\big[\pi^{+}\big(f\hat{D}_{F}^{-1} \big)\circ\pi^{+}\big(f^{-1}(\hat{D}_{F}^{*})^{-1}\big)\big]
&=&Wres[f\hat{D}_{F}^{-1} f^{-1}(\hat{D}_{F}^{*})^{-1}]+\int_{\partial M}\widetilde{\Psi}\nonumber\\
&=&Wres\big[(\hat{D}_{F}^{*}f \hat{D}_{F}f^{-1})^{-1} \big]+\int_{\partial M}\widetilde{\Psi}\nonumber\\
&=&Wres\big[\big(\hat{D}_{F}^{*}\hat{D}_{F}-\hat{D}_{F}^{*}c(df)f^{-1}\big)^{-1} \big]+\int_{\partial M}\widetilde{\Psi}
\end{eqnarray}
where
 \begin{eqnarray}
\widetilde{\Psi}&=&\int_{|\xi'|=1}\int_{-\infty}^{+\infty}\sum_{j,k=0}^{\infty}\sum \frac{(-i)^{|\alpha|+j+k+\ell}}{\alpha!(j+k+1)!}
{\bf{trace}}_{S(TM)\otimes F}\Big[\partial_{x_{n}}^{j}\partial_{\xi'}^{\alpha}\partial_{\xi_{n}}^{k}\sigma_{r}^{+}
(f\hat{D}_{F}^{-1})(x',0,\xi',\xi_{n})\nonumber\\
&&\times\partial_{x_{n}}^{\alpha}\partial_{\xi_{n}}^{j+1}\partial_{x_{n}}^{k}\sigma_{l}(f^{-1}(\hat{D}_{F}^{*})^{-1})(x',0,\xi',\xi_{n})\Big]
\texttt{d}\xi_{n}\sigma(\xi')\texttt{d}x' ,
\end{eqnarray}
and the sum is taken over $r-k+|\alpha|+\ell-j-1=-n,r\leq-1,\ell\leq-1$.

 Let  $\hat{c}(\omega)=\sum_{i}c(e_{i}) \omega(F,g^{F})(e_{i})$
 and $\hat{c}(\omega^{*})=\sum_{i}c(e_{i}) \omega^{*}(F,g^{F})(e_{i})$, then
\begin{eqnarray}\label{eq:3}
&&\hat{D}_{F}^{*}\hat{D}_{F}-\hat{D}_{F}^{*}c(df)f^{-1} \nonumber\\
&=&-g^{ij}\partial_{i}\partial_{j}-2\sigma^{j}_{\wedge^{\ast}(T^{\ast}M)\otimes F}\partial_{j}+\Gamma^{k}\partial_{k}   -c(e_{i})c(df)f^{-1}\partial_{i}
                -\frac{1}{2}\sum_{j}\Big[\hat{c}(\omega^{*})c(e_{j})+c(e_{j})\otimes \hat{c}(\omega) \Big]e_{j}\nonumber\\
                &&-g^{ij}\Big[\partial_{i}(\sigma^{j,e}_{\wedge^{\ast}(T^{\ast}M)\otimes F}) +\sigma^{i}_{\wedge^{\ast}(T^{\ast}M)\otimes F}
                \sigma^{j,e}_{\wedge^{\ast}(T^{\ast}M)\otimes F,e}
                  -\Gamma_{ij}^{k}\sigma_{\wedge^{\ast}(T^{\ast}M)\otimes F}^{k}\Big]\nonumber\\
                &&-\frac{1}{2}\sum_{j}g^{ij}\hat{c}(\omega^{*})c(e_{j})\sigma^{\wedge^{\ast}(T^{\ast}M)\otimes F,e}_{j}
               -\frac{1}{2}\sum_{j}g^{ij}c(e_{j})\otimes e_{j}\big(\hat{c}(\omega)\big) \nonumber\\
                &&-\frac{1}{2}\sum_{j}g^{ij}c(e_{j})\sigma_{j}^{\wedge^{\ast}(T^{\ast}M)\otimes F,e}\otimes \hat{c}(\omega)
                +\frac{1}{4}\hat{c}(\omega^{*})\hat{c}(\omega)\nonumber\\
                               &&+\frac{1}{4}s+\frac{1}{2}\sum_{i\neq j} R^{F,e}(e_{i},e_{j})c(e_{i})c(e_{j})\nonumber\\
&&-\frac{1}{4}\sum_{i}g^{ij}c\sum_{s,t}\omega _{s,t}(e_{i}) \Big[\hat{c}(\omega^{*})c(e_{j})+c(e_{j})\hat{c}(\omega)\Big]c(df)f^{-1}\nonumber\\
&&-\sum_{i}g^{ij}c(e_{j})\sigma^{j,e}_{\wedge^{\ast}(T^{\ast}M)\otimes F,e}c(df)f^{-1}-\frac{1}{2} \hat{c}(e_{j})\omega^{*}(F,g^{F})c(df)f^{-1}.
\end{eqnarray}
In terms of local coordinates $\{\partial_{i}\}$ inducing the  coordinate transformation
 $e_{j}=\sum_{k=1}^{n}\langle e_{j}, \texttt{d}x^{k}\rangle \partial_{k}$, then
\begin{eqnarray}
\omega_{j}&=&\sigma^{j}_{\wedge^{\ast}(T^{\ast}M)}+\sigma^{j,e}_{F}+c(e_{j}) c(df)f^{-1}+\frac{1}{2}\Gamma^{j}\nonumber\\
&&+\frac{1}{4}\Big(\sum_{k=1}^{n}\langle e_{k}, \texttt{d}x^{j}\rangle \hat{c}(\omega^{*}(F,g^{F}))c(e_{k})
                 -\sum_{k=1}^{n}\langle e_{k}, \texttt{d}x^{j}\rangle c(e_{k})\hat{c}(\omega(F,g^{F}))\Big).
\end{eqnarray}
 For a smooth vector field $X\in \Gamma(M,TM)$,  then
 \begin{equation}
\nabla'_{X}=\nabla^{ \wedge^{\ast}(T^{\ast}M)\otimes F,e}_{X}+c(X) c(df)f^{-1}+\frac{1}{4}[\hat{c}(\omega^{*}(F,g^{F}))c(X)-c(X)\hat{c}(\omega(F,g^{F}))].
\end{equation}

Since $E$ is globally defined on $M$, so we can perform
computations of $E$ in normal coordinates. In terms of normal coordinates about $x_{0}$ one has:
$\sigma^{j}_{\wedge^{\ast}(T^{\ast}M)}(x_{0})=0$, $e_{j}\big(c(e_{i})\big)(x_{0})=0$, $\Gamma^{k}(x_{0})=0$.
From (3.12) and (3.13), we obtain
\begin{eqnarray}
E(x_{0})&=&-\frac{1}{4}s-\frac{1}{2}\sum_{i\neq j} R^{F,e}(e_{i},e_{j})c(e_{i})c(e_{j})
         -\frac{1}{16}\sum_{i}\Big[\hat{c}(\omega^{*})c(e_{i})-c(e_{i})\hat{c}(\omega) \Big]^{2}
         -\frac{1}{4}\hat{c}(\omega^{*})\hat{c}(\omega)\nonumber\\
        && -\frac{1}{4}\sum_{j}\nabla_{e_{j}}^{F}\big(\hat{c}(\omega^{*})\big)c(e_{j})
        +\frac{1}{4}\sum_{j}c(e_{j})\nabla_{e_{j}}^{F}\big(\hat{c}(\omega)\big)\nonumber\\
&&+\frac{1}{4}\sum_{i}c(e_{j})\sum_{s,t}\omega _{s,t}(e_{i}) \Big[\hat{c}(\omega^{*})c(e_{j})+c(e_{j})\hat{c}(\omega)\Big]c(df)f^{-1} \nonumber\\
&&+\sum_{i}c(e_{j})\sigma^{j,e}_{\wedge^{\ast}(T^{\ast}M)\otimes F,e}c(df)f^{-1}-\frac{1}{2} \hat{c}(e_{j})\omega^{*}(F,g^{F})c(df)f^{-1}\nonumber\\
&&-\partial_{j}(c(\partial_{j})c(df)f^{-1})\nonumber\\
&&-\frac{1}{2}\big[c(e_{i}) c(df)f^{-1}+\frac{1}{4}\Big(\sum_{k=1}^{n}\langle e_{k}, \texttt{d}x^{j}\rangle \hat{c}(\omega^{*}(F,g^{F}))c(e_{k})
                 -\sum_{k=1}^{n}\langle e_{k}, \texttt{d}x^{j}\rangle c(e_{k})\hat{c}(\omega(F,g^{F}))\Big) \big]c(\partial_{i}) c(df)f^{-1}\nonumber\\
&&-\frac{1}{2}c(\partial_{i}) c(df)f^{-1}\big[c(e_{i}) c(df)f^{-1}
 +\frac{1}{4}\Big(\sum_{k=1}^{n}\langle e_{k}, \texttt{d}x^{j}\rangle \hat{c}(\omega^{*}(F,g^{F}))c(e_{k})
                 -\sum_{k=1}^{n}\langle e_{k}, \texttt{d}x^{j}\rangle c(e_{k})\hat{c}(\omega(F,g^{F}))\Big)\big]\nonumber\\
          &&-\frac{1}{4}c(\partial_{i}) c(df)f^{-1}c(\partial_{i}) c(df)f^{-1},
\end{eqnarray}
From (3.15) and Lemma 2.4 we obtain
\begin{eqnarray}
\texttt{Tr}(E(x_{0}))
        &=&{\bf{Tr}}\Big[-\frac{1}{4}s +\frac{n}{16}[\hat{c}(\omega^{*})-\hat{c}(\omega)]^{2}
         -\frac{1}{4}\hat{c}(\omega^{*})\hat{c}(\omega) -\frac{1}{4}\sum_{j}\nabla_{e_{j}}^{F}\big(\hat{c}(\omega^{*})\big)c(e_{j})\nonumber\\
       && +\frac{1}{4}\sum_{j}c(e_{j})\nabla_{e_{j}}^{F}\big(\hat{c}(\omega)\big) \Big]\nonumber\\
&& +{\bf{Tr}} \Big[-\partial_{j}(c(\partial_{j})c(df)f^{-1})
-\frac{5}{4} c(e_{i}) c(df)f^{-1} c(\partial_{i}) c(df)f^{-1}\Big]\nonumber\\
 &=&{\bf{Tr}}\Big[-\frac{1}{4}s +\frac{n}{16}[\hat{c}(\omega^{*})-\hat{c}(\omega)]^{2}
         -\frac{1}{4}\hat{c}(\omega^{*})\hat{c}(\omega) -\frac{1}{4}\sum_{j}\nabla_{e_{j}}^{F}\big(\hat{c}(\omega^{*})\big)c(e_{j})\nonumber\\
       && +\frac{1}{4}\sum_{j}c(e_{j})\nabla_{e_{j}}^{F}\big(\hat{c}(\omega)\big) \Big]\nonumber\\
&& +4f^{-1}\Delta(f)+8\langle grad_{M}(f) , grad_{M}(f^{-1}) \rangle-5f^{-2}\big[ |grad_{M}(f)|^{2}+2\Delta(f)\big].
\end{eqnarray}
Hence we conclude that
\begin{thm}
For even $n$-dimensional oriented  compact Riemainnian manifolds without boundary,
the following equality holds:
\begin{eqnarray}
&& Wres(f\hat{D}_{F}^{-1} \circ f^{-1}(\hat{D}_{F}^{*})^{-1})^{(\frac{n-2}{2})}\nonumber\\
&=&\frac{(2\pi)^{\frac{n}{2}}}{(\frac{n}{2}-2)!}\int_{M}{\bf{Tr}}
 \Big[-\frac{s}{12} +\frac{n}{16}[\hat{c}(\omega^{*})-\hat{c}(\omega)]^{2}
         -\frac{1}{4}\hat{c}(\omega^{*})\hat{c}(\omega)\nonumber\\
       && -\frac{1}{4}\sum_{j}\nabla_{e_{j}}^{F}\big(\hat{c}(\omega^{*})\big)c(e_{j})
        +\frac{1}{4}\sum_{j}c(e_{j})\nabla_{e_{j}}^{F}\big(\hat{c}(\omega)\big) \Big]\nonumber\\
 && +4f^{-1}\Delta(f)+8\langle grad_{M}(f) , grad_{M}(f^{-1}) \rangle-5f^{-2}\big[ |grad_{M}(f)|^{2}+2\Delta(f)\big]\texttt{d}vol_{M}.
\end{eqnarray}
\end{thm}
Locally we can use Theorem 3.3 to compute the interior term of (3.10), then
\begin{eqnarray}
&&\int_{M}\int_{|\xi|=1}{\bf{trace}}_{\wedge^{\ast}(T^{\ast}M)\otimes F}
  [\sigma_{-4}((f\hat{D}_{F}^{-1} \circ f^{-1}(\hat{D}_{F}^{*})^{-1})]\sigma(\xi)\texttt{d}x\nonumber\\
&=&4\pi^{2}\int_{M}{\bf{Tr}}
 \Big[-\frac{s}{12} +\frac{1}{4}[\hat{c}(\omega^{*})-\hat{c}(\omega)]^{2}
         -\frac{1}{4}\hat{c}(\omega^{*})\hat{c}(\omega)-\frac{1}{4}\sum_{j}\nabla_{e_{j}}^{F}\big(\hat{c}(\omega^{*})\big)c(e_{j})\nonumber\\
       &&  +\frac{1}{4}\sum_{j}c(e_{j})\nabla_{e_{j}}^{F}\big(\hat{c}(\omega)\big) \Big]\nonumber\\
&& +4f^{-1}\Delta(f)+8\langle grad_{M}(f) , grad_{M}(f^{-1}) \rangle-5f^{-2}\big[ |grad_{M}(f)|^{2}+2\Delta(f)\big]\texttt{d}vol_{M}.
\end{eqnarray}

So we only need to compute $\int_{\partial M} \widetilde{\Psi}$.  In the local coordinates $\{x_{i}; 1\leq i\leq n\}$ and the fixed orthonormal frame
$\{\widetilde{e_{1}},\cdots, \widetilde{e_{n}}\}$, the connection matrix $(\omega_{s,t})$ is defined by
\begin{equation}
\widetilde{\nabla}(\widetilde{e_{1}},\cdots, \widetilde{e_{n}})=(\widetilde{e_{1}},\cdots, \widetilde{e_{n}})(\omega_{s,t}).
\end{equation}

 Let $M$ be a $4$-dimensional compact oriented Riemannian manifold with boundary $\partial M$ and the metric of (2.6).
$ \hat{D}_{F} =d^{F}+\delta^{F}:~C^{\infty}(M,\wedge^{\ast}(T^{\ast}M)\otimes F)\rightarrow C^{\infty}(M,\wedge^{\ast}(T^{\ast}M)\otimes F)$
is the twisted  signature operator. Take the coordinates and
the orthonormal frame as in Section 2.
 Let $\epsilon (\widetilde{e_j*} ),~\iota (\widetilde{e_j*} )$ be the exterior and interior multiplications respectively. Write
 \begin{equation}
c(\widetilde{e_j})=\epsilon (\widetilde{e_j*} )-\iota (\widetilde{e_j*} );~~
\hat{c}(\widetilde{e_j})=\epsilon (\widetilde{e_j*} )+\iota (\widetilde{e_j*} ).
\end{equation}
  We'll compute ${\rm tr}_{\wedge^*(T^*M)\otimes F}$ in the frame $\{e^{\ast}_{i_1}\wedge\cdots\wedge
e^{\ast}_{i_k}|~1\leq i_1<\cdots<i_k\leq 4\}.$ By (3.2) in \cite{Wa3}, we have
\begin{eqnarray}
 \hat{D}_{F}&=&d^{F}+\delta^{F}=\sum_{i=1}^{n}c(e_{i})\nabla^{e}_{e_{i}}-\frac{1}{2}\sum_{i=1}^{n}\hat{c}(e_{i})\omega(F,g^{F})(e_{i})\nonumber\\
    &=&\sum_{i=1}^{n}c(e_{i})\Big(\nabla_{e_{i}}^{\wedge^{\ast}(T^{\ast}M)}\otimes id_{F}+id_{\wedge^{\ast}(T^{\ast}M)} \otimes \nabla^{F,e}_{e_{i}} \Big)
    -\frac{1}{2}\sum_{i=1}^{n}\hat{c}(e_{i})\omega(F,g^{F})(e_{i})\nonumber\\
    &=&\sum^n_{i=1}c(\widetilde{e_i})\Big[\widetilde{e_i}+\frac{1}{4}\sum_{s,t}\omega_{s,t}
(\widetilde{e_i})[\hat{c}(\widetilde{e_s})\hat{c}(\widetilde{e_t})-c(\widetilde{e_s})c(\widetilde{e_t})]\otimes id_{F}
  \nonumber\\
   &&+id_{\wedge^{\ast}(T^{\ast}M)}\otimes \sigma^{F,e}_{i}\Big]-\frac{1}{2}\sum_{i=1}^{n}\hat{c}(e_{i})\omega(F,g^{F})(e_{i}),\\
\hat{D}^{*}_{F}&=&\sum^n_{i=1}c(\widetilde{e_i})\Big[\widetilde{e_i}+\frac{1}{4}\sum_{s,t}\omega_{s,t}
(\widetilde{e_i})[\hat{c}(\widetilde{e_s})\hat{c}(\widetilde{e_t})-c(\widetilde{e_s})c(\widetilde{e_t})]\otimes id_{F}
  \nonumber\\
   &&+id_{\wedge^{\ast}(T^{\ast}M)}\otimes \sigma^{F,e}_{i}\Big]-\frac{1}{2}\sum_{i=1}^{n}\hat{c}(e_{i})\omega^{*}(F,g^{F})(e_{i}).
\end{eqnarray}
Then
\begin{eqnarray}
\sigma_1(\hat{D}_{F})&=&\sigma_1(\hat{D}^{*}_{F})=\sqrt{-1}c(\xi);\\
\sigma_0(\hat{D}_{F})&=&\sum^n_{i=1}c(\widetilde{e_i})\Big[\frac{1}{4}\sum_{s,t}\omega_{s,t}
(\widetilde{e_i})[\hat{c}(\widetilde{e_s})\hat{c}(\widetilde{e_t})-c(\widetilde{e_s})c(\widetilde{e_t})]\otimes \texttt{id}_{F}
  +id_{\wedge^{\ast}(T^{\ast}M)}\otimes \sigma^{F,e}_{i}\Big]\nonumber\\
   &&-\frac{1}{2}\sum_{i=1}^{n}\hat{c}(e_{i})\omega(F,g^{F})(e_{i});\\
\sigma_0(\hat{D}^{*}_{F})&=&\sum^n_{i=1}c(\widetilde{e_i})\Big[\frac{1}{4}\sum_{s,t}\omega_{s,t}
(\widetilde{e_i})[\hat{c}(\widetilde{e_s})\hat{c}(\widetilde{e_t})-c(\widetilde{e_s})c(\widetilde{e_t})]\otimes \texttt{id}_{F}
  +id_{\wedge^{\ast}(T^{\ast}M)}\otimes \sigma^{F,e}_{i}\Big]\nonumber\\
   &&-\frac{1}{2}\sum_{i=1}^{n}\hat{c}(e_{i})\omega^{*}(F,g^{F})(e_{i}).
\end{eqnarray}
By the composition formula of pseudodifferential operators in Section 2.2.1 of \cite{Wa3}, we have
 \begin{lem}\label{le:31}
The symbol of the  twisted signature operators  $\hat{D}^{*}_{F}, \hat{D}_{F}$ as follows:
\begin{eqnarray}
\sigma_{-1}((\hat{D}_{F})^{-1})&=&\sigma_{-1}((\hat{D}^{*}_{F})^{-1})=\frac{\sqrt{-1}c(\xi)}{|\xi|^{2}}; \\
\sigma_{-2}((\hat{D}_{F})^{-1})&=&\frac{c(\xi)\sigma_{0}(\hat{D}_{F})c(\xi)}{|\xi|^{4}}+\frac{c(\xi)}{|\xi|^{6}}\sum_{j}c(\texttt{d}x_{j})
\Big[\partial_{x_{j}}(c(\xi))|\xi|^{2}-c(\xi)\partial_{x_{j}}(|\xi|^{2})\Big];\\
\sigma_{-2}((\hat{D}^{*}_{F})^{-1})&=&\frac{c(\xi)\sigma_{0}(\hat{D}^{*}_{F})c(\xi)}{|\xi|^{4}}+\frac{c(\xi)}{|\xi|^{6}}\sum_{j}c(\texttt{d}x_{j})
\Big[\partial_{x_{j}}(c(\xi))|\xi|^{2}-c(\xi)\partial_{x_{j}}(|\xi|^{2})\Big].
\end{eqnarray}
\end{lem}

Since $ \widetilde{\Psi}$ is a global form on $\partial M$, so for any fixed point $x_{0}\in\partial M$, we can choose the normal coordinates
$U$ of $x_{0}$ in $\partial M$(not in $M$) and compute $ \widetilde{\Psi}(x_{0})$ in the coordinates $\widetilde{U}=U\times [0,1)$ and the metric
$\frac{1}{h(x_{n})}g^{\partial M}+\texttt{d}x _{n}^{2}$. The dual metric of $g^{\partial M}$ on $\widetilde{U}$ is
$\frac{1}{\tilde{h}(x_{n})}g^{\partial M}+\texttt{d}x _{n}^{2}.$ Write
$g_{ij}^{M}=g^{M}(\frac{\partial}{\partial x_{i}},\frac{\partial}{\partial x_{j}})$;
$g^{ij}_{M}=g^{M}(d x_{i},dx_{j})$, then

\begin{equation}
[g_{i,j}^{M}]=
\begin{bmatrix}\frac{1}{h( x_{n})}[g_{i,j}^{\partial M}]&0\\0&1\end{bmatrix};\quad
[g^{i,j}_{M}]=\begin{bmatrix} h( x_{n})[g^{i,j}_{\partial M}]&0\\0&1\end{bmatrix},
\end{equation}
and
\begin{equation}
\partial_{x_{s}} g_{ij}^{\partial M}(x_{0})=0,\quad 1\leq i,j\leq n-1;\quad g_{i,j}^{M}(x_{0})=\delta_{ij}.
\end{equation}

Let $\{e_{1},\cdots, e_{n-1}\}$ be an orthonormal frame field in $U$ about $g^{\partial M}$ which is parallel along geodesics and
$e_{i}=\frac{\partial}{\partial x_{i}}(x_{0})$, then $\{\widetilde{e_{1}}=\sqrt{h(x_{n})}e_{1}, \cdots,
\widetilde{e_{n-1}}=\sqrt{h(x_{n})}e_{n-1},\widetilde{e_{n}}=\texttt{d}x_{n}\}$ is the orthonormal frame field in $\widetilde{U}$ about $g^{M}.$
Locally $\wedge^{\ast}(T^{\ast}M)|\widetilde{U}\cong \widetilde{U}\times\wedge^{*}_{C}(\frac{n}{2}).$ Let $\{f_{1},\cdots,f_{n}\}$ be the orthonormal basis of
$\wedge^{*}_{C}(\frac{n}{2})$. Take a spin frame field $\sigma: \widetilde{U}\rightarrow Spin(M)$ such that
$\pi\sigma=\{\widetilde{e_{1}},\cdots, \widetilde{e_{n}}\}$ where $\pi: Spin(M)\rightarrow O(M)$ is a double covering, then
$\{[\sigma, f_{i}], 1\leq i\leq 4\}$ is an orthonormal frame of $\wedge^{\ast}(T^{\ast}M)|_{\widetilde{U}}.$ In the following,
since the global form $ \widetilde{\Psi}$
is independent of the choice of the local frame, so we can compute $\texttt{tr}_{\wedge^{\ast}(T^{\ast}M)}$ in the frame
$\{[\sigma, f_{i}], 1\leq i\leq 4\}$.
Let $\{E_{1},\cdots,E_{n}\}$ be the canonical basis of $R^{n}$ and
$c(E_{i})\in cl_{C}(n)\cong \texttt{Hom}(\wedge^{*}_{C}(\frac{n}{2}),\wedge^{*}_{C}(\frac{n}{2}))$ be the Clifford action. By \cite{Wa3}, then

\begin{equation}
c(\widetilde{e_{i}})=[(\sigma,c(E_{i}))]; \quad c(\widetilde{e_{i}})[(\sigma, f_{i})]=[\sigma,(c(E_{i}))f_{i}]; \quad
\frac{\partial}{\partial x_{i}}=[(\sigma,\frac{\partial}{\partial x_{i}})],
\end{equation}
then we have $\frac{\partial}{\partial x_{i}}c(\widetilde{e_{i}})=0$ in the above frame. By Lemma 2.2 in \cite{Wa3}, we have

\begin{lem}\label{le:32}
With the metric $g^{M}$ on $M$ near the boundary
\begin{eqnarray}
\partial_{x_j}(|\xi|_{g^M}^2)(x_0)&=&\left\{
       \begin{array}{c}
        0,  ~~~~~~~~~~ ~~~~~~~~~~ ~~~~~~~~~~~~~{\rm if }~j<n; \\[2pt]
       h'(0)|\xi'|^{2}_{g^{\partial M}},~~~~~~~~~~~~~~~~~~~~~{\rm if }~j=n.
       \end{array}
    \right. \\
\partial_{x_j}[c(\xi)](x_0)&=&\left\{
       \begin{array}{c}
      0,  ~~~~~~~~~~ ~~~~~~~~~~ ~~~~~~~~~~~~~{\rm if }~j<n;\\[2pt]
\partial x_{n}(c(\xi'))(x_{0}), ~~~~~~~~~~~~~~~~~{\rm if }~j=n,
       \end{array}
    \right.
\end{eqnarray}
where $\xi=\xi'+\xi_{n}\texttt{d}x_{n}$
\end{lem}
Then an application of Lemma 2.3 in \cite{Wa3} shows
\begin{lem}
The symbol of the  twisted signature operators  $\hat{D}^{*}_{F},   \hat{D}_{F}$
\begin{eqnarray}
\sigma_{0}(\hat{D}^{*}_{F})&=&-\frac{3}{4}h'(0)c(dx_n)
+\frac{1}{4}h'(0)\sum^{n-1}_{i=1}c(\widetilde{e_i})\hat{c}(\widetilde{e_n})\hat{c}(\widetilde{e_i})(x_0)\otimes id_{F}  \nonumber\\
  &&+\sum^n_{i=1}c(\widetilde{e_i})\sigma^{F,e}_{i}
   -\frac{1}{2}\sum_{i=1}^{n}\hat{c}(e_{i})\omega^{*}(F,g^{F})(e_{i});\\
 \sigma_{0}(\hat{D}_{F})&=&-\frac{3}{4}h'(0)c(dx_n)
+\frac{1}{4}h'(0)\sum^{n-1}_{i=1}c(\widetilde{e_i})\hat{c}(\widetilde{e_n})\hat{c}(\widetilde{e_i})(x_0)\otimes id_{F}  \nonumber\\
  &&+\sum^n_{i=1}c(\widetilde{e_i})\sigma^{F,e}_{i}
   -\frac{1}{2}\sum_{i=1}^{n}\hat{c}(e_{i})\omega(F,g^{F})(e_{i})
\end{eqnarray}
\end{lem}

Now we can compute $ \widetilde{\Psi}$ (see formula (3.11) for definition of $ \widetilde{\Psi}$), since the sum is taken over $-r-\ell+k+j+|\alpha|=3,
 \ r, \ell\leq-1$, then we have the following five cases:

\textbf{Case a(I)}: \ $r=-1, \ \ell=-1, \ k=j=0, \ |\alpha|=1$

From (3.11) we have
 \begin{eqnarray}
&&{\rm case~a)~I)}\nonumber\\
&=&-\int_{|\xi'|=1}\int^{+\infty}_{-\infty}\sum_{|\alpha|=1}
{\rm trace} \Big[\partial^\alpha_{\xi'}\pi^+_{\xi_n}\sigma_{-1}( \hat{D}_{F}^{-1})\times
\partial^\alpha_{x'}\partial_{\xi_n}\sigma_{-1}(( \hat{D}^{*}_{F})^{-1})\Big](x_0)\texttt{d}\xi_n\sigma(\xi')\texttt{d}x'\nonumber\\
&&-f\sum_{j<n}\partial_{j}(f^{-1})\int_{|\xi'|=1}\int^{+\infty}_{-\infty}\sum_{|\alpha|=1}
{\rm trace} \Big[\partial^\alpha_{\xi'}\pi^+_{\xi_n}\sigma_{-1}( \hat{D}_{F}^{-1})\times
\partial_{\xi_n}\sigma_{-1}(( \hat{D}^{*}_{F})^{-1})\Big](x_0)\texttt{d}\xi_n\sigma(\xi')\texttt{d}x'\nonumber\\
&=&0.
\end{eqnarray}
And similarly we get

{\bf case a)~II)}~$r=-1,~l=-1~k=|\alpha|=0,~j=1$
 \begin{eqnarray}
&&{\rm case \ a)~II)}\nonumber\\
&=&-\frac{1}{2}\int_{|\xi'|=1}\int^{+\infty}_{-\infty} {\rm
trace} \Big[\partial_{x_n}\pi^+_{\xi_n}\sigma_{-1}(  \hat{D} _{F} ^{-1})\times
\partial_{\xi_n}^2\sigma_{-1}(( \hat{D}^{*}_{F})^{-1})\Big](x_0)\texttt{d}\xi_n\sigma(\xi')\texttt{d}x'\nonumber\\
&&-f^{-1}\partial_{x_n}(f)\frac{1}{2}\int_{|\xi'|=1}\int^{+\infty}_{-\infty} {\rm
trace} \Big[\partial_{x_n} \sigma_{-1}(  \hat{D} _{F} ^{-1})\times
\partial_{\xi_n}^2\sigma_{-1}(( \hat{D}^{*}_{F})^{-1})\Big](x_0)\texttt{d}\xi_n\sigma(\xi')\texttt{d}x'\nonumber\\
&=&-\frac{3l}{2}\pi h'(0)\Omega_3dx'-\frac{\pi i}{2}\Omega_{3}f^{-1}\partial_{x_n}(f)\texttt{d}x'.
\end{eqnarray}

 {\bf case a)~III)}~$r=-1,~l=-1~j=|\alpha|=0,~k=1$
 \begin{eqnarray}
 &&{\rm case~ a)~III)}\nonumber\\
&=&-\frac{1}{2}\int_{|\xi'|=1}\int^{+\infty}_{-\infty}
{\rm trace} \Big[\partial_{\xi_n}\pi^+_{\xi_n}\sigma_{-1}( \hat{D}_{F}^{-1})\times
\partial_{\xi_n}\partial_{x_n}\sigma_{-1}(( \hat{D}^{*}_{F})^{-1})\Big](x_0)\texttt{d}\xi_n\sigma(\xi')\texttt{d}x'\nonumber\\
&&-f\partial_{x_{n}}(f^{-1})\Big[\partial_{\xi_n}\pi^+_{\xi_n}\sigma_{-1}( \hat{D}_{F}^{-1})\times
\partial_{\xi_n}\partial_{x_n}\sigma_{-1}(( \hat{D}^{*}_{F})^{-1})\Big](x_0)\texttt{d}\xi_n\sigma(\xi')\texttt{d}x'\nonumber\\
&=&\frac{3l}{2}\pi h'(0)\Omega_{3}\texttt{d}x'+\frac{\pi i}{2}\Omega_{3}f\partial_{x_{n}}(f^{-1})\texttt{d}x'.
\end{eqnarray}

 {\bf case b)}~$r=-2,~l=-1,~k=j=|\alpha|=0$
\begin{eqnarray}
{\rm case~ b)}&=&-i\int_{|\xi'|=1}\int^{+\infty}_{-\infty}
{\rm trace} \Big[\pi^+_{\xi_n}\sigma_{-2}(f \hat{D}_{F}^{-1})\times
\partial_{\xi_n}\sigma_{-1}((f \hat{D}^{*}_{F})^{-1})\Big](x_0)\texttt{d}\xi_n\sigma(\xi')\texttt{d}x' \nonumber\\
&=&\Big[\frac{9}{2}lh'(0)-4 \texttt{tr}_{F}[\sigma^{F,e}_{n}]\Big]\pi \Omega_{3}\texttt{d}x'..
\end{eqnarray}

{\bf  case c)}~$r=-1,~l=-2,~k=j=|\alpha|=0$
\begin{eqnarray}
{\rm case~ c)}&=&-i\int_{|\xi'|=1}\int^{+\infty}_{-\infty}{\rm trace}\Big [\pi^+_{\xi_n}\sigma_{-1}(f \hat{D}_{F}^{-1})\times
\partial_{\xi_n}\sigma_{-2}(f^{-1}( \hat{D}^{*}_{F})^{-1})\Big](x_0)d\xi_n\sigma(\xi')dx'\nonumber\\
&=&\Big[-\frac{9}{2}lh'(0)+4\texttt{tr}_{F}[\sigma^{F,e}_{n}]\Big]\pi \Omega_{3}\texttt{d}x'.
\end{eqnarray}
We note that $\texttt{dim} S(TM)=4$, now $\widetilde{\Psi}$  is the sum of the \textbf{case (a, b, c)}, so
\begin{equation}
\sum \textbf{case a, b , c}=\frac{\pi i}{2}\Omega_{3} \big[f\partial_{x_{n}}(f^{-1})- f^{-1}\partial_{x_{n}}(f)\big]\texttt{d}x'
\end{equation}
Hence we conclude that
\begin{thm}
 Let M be a 4-dimensional compact manifolds   with the boundary $\partial M$, for Perturbations of  twisted signature Operators $\hat{D}_{F}$, then
 \begin{eqnarray}
&&\widetilde{Wres}[\pi^{+}( \hat{D}_{F}^{*})^{-1} \circ\pi^{+} \hat{D}_{F}^{-1}]\nonumber\\
&=&4\pi^{2}\int_{M}\big\{{\bf{Tr}}
 \Big[-\frac{s}{12} +\frac{1}{4}[\hat{c}(\omega^{*})-\hat{c}(\omega)]^{2}
         -\frac{1}{4}\hat{c}(\omega^{*})\hat{c}(\omega)-\frac{1}{4}\sum_{j}\nabla_{e_{j}}^{F}\big(\hat{c}(\omega^{*})\big)c(e_{j})\nonumber\\
       &&  +\frac{1}{4}\sum_{j}c(e_{j})\nabla_{e_{j}}^{F}\big(\hat{c}(\omega)\big) \Big]\nonumber\\
&& +4f^{-1}\Delta(f)+8\langle grad_{M}(f) , grad_{M}(f^{-1}) \rangle-5f^{-2}\big[ |grad_{M}(f)|^{2}+2\Delta(f)\big]\big\}\texttt{d}vol_{M}\nonumber\\
&&+\int_{\partial_{ M}}\frac{\pi i}{2}\Omega_{3} \big[f\partial_{x_{n}}(f^{-1})- f^{-1}\partial_{x_{n}}(f)\big]\texttt{d}x'
\end{eqnarray}
where  $s$ is the scalar curvature.
\end{thm}

\section*{ Acknowledgements}
This work was supported by the National Natural Science Foundation of China No. 11501414 and No. 11771070.
 The authors also thank the referee for his (or her) careful reading and helpful comments.

\end{document}